\documentclass[a4paper,11pt]{article}

\usepackage{palatino}
\usepackage{mathpazo}

\usepackage{mathrsfs}
\usepackage{amssymb,amsmath,amsthm,url}
\usepackage{verbatim}
\usepackage[small]{titlesec}

\usepackage{enumitem}

\usepackage{pstricks}
\usepackage{graphicx, pst-plot, pst-node, pst-text, pst-tree}

\theoremstyle{plain} 
\newtheorem*{mainthm}{Main Theorem}

\newtheorem{lemma}{Lemma}[section]
\newtheorem{prop}[lemma]{Proposition}
\newtheorem{claim}{Claim}
\newtheorem{thm}[lemma]{Theorem}

\theoremstyle{definition}

\newtheorem{example}[lemma]{Example}

\theoremstyle{remark}

\newcommand{\st}{\::\:}
\newcommand{\rrestriction}{\!\!\restriction}

\title{On hopfian cofinite subsemigroups}

\author{Victor Maltcev and N. Ru\v{s}kuc}

\begin{document}

\maketitle

\begin{abstract}
If a finitely generated semigroup $S$ has a hopfian 
(meaning: every surjective endomorphism is an automorphism) cofinite subsemigroup $T$
then $S$ is hopfian too. 
This no longer holds if $S$ is not finitely generated. There exists a finitely generated hopfian semigroup $S$ with a non-hopfian subsemigroup $T$ such that $S\setminus T$ has size $1$.

\bigskip


\bigskip

\end{abstract}

\section{Introduction and the statement of main result} 
\label{sec_intro}

An algebraic structure $A$ is said to be \emph{hopfian} if no proper quotient of $A$ is isomorphic to $A$, or, equivalently, if
every surjective endomorphism of $A$ is an automorphism.
Hopficity is clearly a finiteness condition (i.e. all finite algebraic structures are hopfian), and so the question arises of its preservation under substructures that are in some sense `large', or extensions that are in some sense `small'.

The property was introduced into literature by Hopf \cite{ hopf31} who asked if every finitely generated group  was hopfian.
The group defined by the presentation
$\langle a,b\:|\: a^{-1}b^2a=\nobreak b^3\rangle$
 is a simple (but not first) counter-example; see \cite{BS}.
In the same paper the authors show that the group $\langle a,b \:|\: a^{-1}b^{12}a=b^{18}\rangle$ is hopfian, but
contains a normal non-hopfian subgroup of index $6$. 
In particular,
hopficity is not preserved by subgroups of finite index, even in the
finitely generated case. 
By way of contrast, Hirshon \cite{Hirshon}  proved
that if $H$ is a hopfian  subgroup of
finite index in a finitely generated group $G$, then $G$ is hopfian
as well. 
To the best of our knowledge, it is still open if
the same statement holds without finite generation assumption.

It has been known for a while that in semigroups processes of taking cofinite subsemigroups, and conversely extending by finite sets, display strong analogies with taking subgroups and extensions of finite index in group theory. To emphasise this analogy the term \emph{Rees index} has been introduced; it is simply the size of the complement of the subsemigroup in its parent semigroup.
Even though to have finite Rees index is a fairly restrictive property, it occurs naturally in semigroups (for instance all ideals in the additive semigroup of positive integers are of finite Rees index), and it has provided a fertile ground for research,  throwing up a few surprises along the way. In \cite{Nik} it is proved that the main combinatorial finiteness conditions, such as finite generation, presentability, and solvability of the word problem, are all preserved both under finite Rees subsemigroups and extensions.
The proof for finite presentability is surprisingly complicated, and to date no fundamentally simpler proof has been found.
The related question of the existence of finite complete rewriting system has been settled only very recently, see \cite{wongta}.
Some related cohomological finiteness conditions are considered in \cite{malheiro09} and \cite{wang98}, 
residual finiteness is treated in \cite{Rick}, and some surprising behaviour in relation to the ideal structure is exhibited in \cite{James}.

The purpose of this paper is to demonstrate that the situation for semigroups in relation to the hopfian property is analogous to the above described situation for groups.
Specifically, we prove:

\begin{mainthm}\label{thm_main}
Let $S$ be a finitely generated semigroup, and let $T$ be a subsemigroup with $S\setminus T$ finite.
If $T$ is hopfian then $S$ is hopfian as well.
\end{mainthm}

Accompanying the Main Theorem are two examples, establishing the following:

\begin{itemize}[partopsep=0mm,topsep=2mm,itemsep=1mm,partopsep=0mm]
\item
if finite generation assumption is removed, the Main Theorem no longer holds; and
\item
hopficity is not necessarily inherited by cofinite subsemigroups, not even in the finitely generated case.
\end{itemize}

The Main Theorem is proved in Section \ref{sec_proof}, although the brunt of the work goes into proving a result concerning cofinite subsemigroups and endomorphisms in Section \ref{sec_reesendo}.
The accompanying examples are exhibited before and after the proof, in Sections \ref{sec_introex} and \ref{sec_concludeex} respectively.
The final section contains some further commentary and open problems.

\section{An introductory example}
\label{sec_introex}

The purpose of this section is to show that, without adding the finite generation assumption, hopficity is not preserved by either finite Rees index extensions or subsemigroups.

We begin by defining a family of isomorphic semigroups $T_i=\langle b_i\:|\: b_i^2=b_i^4\rangle$, $i\in\mathbb{N}$.
Form their union $T=\bigcup_{i\in \mathbb{N}} T_i$, and extend the multiplication defined on each $T_i$ to a multiplication on the whole of $T$ by letting
$xy=yx=y$ for any $x\in T_i$, $y\in T_j$, $i<j$.
It is easy to see that this turns $T$ into a semigroup.

Further, let $F$ be the semigroup $\langle a \:|\: a^5=a^2\rangle$, let $S=T\cup F$, and extend the multiplication on $T$ and $F$ to a multiplication on the whole of $S$ by $xy=yx=y$ for $x\in F$, $y\in T$.
Again, this turns $S$ into a semigroup.
Finally, let $S^1$ be the semigroup $S$ with an identity adjoined to it.
Clearly we have $T\leq T^1\leq S^1$, a sequence of finite Rees index extensions.

\begin{prop}
\label{example1}
The semigroups $S^1$ and $T$ are hopfian, while the semigroup $T^1$ is not. Hence, hopficity is preserved by neither finite Rees index extensions nor subsemigroups
\end{prop}

\begin{proof}
\textit{$T$ is hopfian.}
Let $\phi:T\rightarrow T$ be a surjective endomorphism.
Since $b_1$ is the only indecomposable element of $T$ (in the sense that it cannot be written as a product of any two elements of $T$),
we must have $b_1\phi^{-1}=\{ b_1\}$.
In particular $\phi\rrestriction_{T_1}$ is the identity mapping.
The set of elements on which $b_1$ acts trivially (meaning $b_1x=xb_1=x$) is precisely $T\setminus T_1$,
so $\phi$ maps this set onto itself.
But clearly $T\setminus T_1$ is a subsemigroup isomorphic to $T$, and an inductive argument shows that $\phi$ is in fact the identity mapping. Thus $T$ indeed is hopfian.
\medskip

\textit{$T^1$ is not hopfian.}
Indeed, a routine verification shows that the mapping $b_1\mapsto 1$, $b_{n+1}\mapsto b_n$ ($n\in \mathbb{N}$)
extends to a surjective, non-injective endomorphism of $T^1$.
\medskip

\textit{$S^1$ is hopfian.}
Let $\phi:S^1\rightarrow S^1$ be a surjective endomorphism.
Clearly, $1\phi=1$.
Note that $a$ is the only element $x\in S^1$ such that $\langle x\rangle$ is not a group and $x^5=x^2$.
It follows that $a\phi^{-1}=\{a\}$.
Similarly, $b_i$ ($i\in\mathbb{N}$) are the only elements $x\in S^1$ such that $\langle x\rangle$ is not a group and $x^4=x^2$.
Hence $\phi$ maps $\langle b_1,b_2,\dots\rangle=T$ onto itself. We have already proved that $T$ is hopfian, and so it follows that
$\phi$ is a bijection, and so $S^1$ is hopfian, as required.
\end{proof}

\section{Cofinite subsemigroups and endomorphisms}
\label{sec_reesendo}

The following result will be of crucial importance in the proof of the Main Theorem.

\begin{thm}
\label{thm1}
For every endomorphism $\phi$ of a finitely generated semigroup $S$ and every proper cofinite subsemigroup $T$  we have $T\phi\neq S$.
\end{thm}

\begin{proof}
Suppose to the contrary that $T\phi=S$.
Let $F=S\setminus T$, a finite set.

\begin{claim}
\label{lemma3.3}
There exists $N\geq 1$ such that for every $f\in F$ at least one of the following holds:
$f\phi^{tN}\in T$ for all $t\geq 1$, or
$f\phi^N=f\phi^{2N}$.
\end{claim}

\begin{proof}
Consider the following two sets:
\begin{eqnarray*}
&&F_\infty=\{ f\in F \st f\phi^k\in F \mbox{ for infinitely many } k\geq 1\},\\
&&F_0=\{f\in F \st f\phi^k\in F \mbox{ for only finitely many } k\geq 1\},
\end{eqnarray*}
which clearly partition $F$.
The following assertions follow easily from finiteness of $F$:
First of all, for every $f\in F_\infty$ its orbit $O(f)=\{ f\phi^k\st k\geq 0\}$ is finite.
Then $O(F_\infty)=\bigcup_{f\in F_\infty} O(f)$ is finite too, and so there exists $p\geq 1$ such that
$\phi^p\rrestriction_{O(F_\infty)}$ is an idempotent.
Finally, there exists $q\geq 1$ such that $F_0\phi^k\subseteq T$ for all $k\geq q$.
Any number $N$ which is greater than $q$ and is a multiple of $p$ will satisfy the conditions of the claim.
\end{proof}

Let us denote the mapping $\phi^N$ by $\pi$.
From the assumption that $T\phi=S$ it follows that $T\pi=S$ as well.
For every $k\geq 0$ let
\begin{equation}
\label{eq1a}
A_k=F\pi^{-k} ,
\end{equation}
and let
\begin{equation}
\label{eq1}
B=T\setminus \bigcup_{k\geq 0} A_k =\{ t\in T\st t\pi^k\in T \mbox{ for all } k\geq 0\}.
\end{equation}
From \eqref{eq1} and $T\pi=S$ it is immediately clear that $B$ is a subsemigroup of $T$ (or possibly empty) and that $B\pi=B$.
Furthermore, from Claim \ref{lemma3.3} it follows that for every $f\in F$ we have:
\begin{equation}
\label{eq2}
f\pi\in B \mbox{ or } f\pi^2=f\pi\in F.
\end{equation}
From \eqref{eq1a}, \eqref{eq1}, \eqref{eq2} it follows that
\begin{equation}
\label{eq3}
A_k\pi^l\subseteq  B\cup F \ (l\geq k \geq 0).
\end{equation}

We now start using finite generation of $S$. We remark that by \cite[Theorem 1.1]{Nik} 
this is equivalent to $T$ being finitely generated.

\begin{claim}
\label{lemma3.4}
There exists a finite set $Y\subseteq B$ such that the set $Y\cup F$ generates $S$.
\end{claim}

\begin{proof}
Let $X$ be any finite generating set for $T$.
Since $X$ is finite there must exist $k\geq 0$ such that $X\subseteq B\cup A_1\cup\dots\cup A_k$.
Note that $X\cup F$ generates $S$.
Since $\pi$ is onto, the set $(X\cup F)\pi^k$ also generates $S$.
But, using \eqref{eq3}, we have $(X\cup F)\pi^k\subseteq B\cup F$.
\end{proof}

Let $U=\{ f\in F\st f\pi=f\}$; clearly $U=F\cap F\pi$ by \eqref{eq2}.
Note that
$$
(Y\cup F)\pi=Y\pi \cup (T\cap F\pi) \cup (F\cap F\pi)=Z\cup U,
$$
where
$Z=Y\pi \cup (T\cap F\pi)\subseteq B$.
So we have:

\begin{claim}
\label{lemma3.5}
There exists a finite set $Z\subseteq B$ such that $Z\cup U$ generates $S$.
\end{claim}

Now let $V=T\cap (U^2\cup U^3)$.

\begin{claim}
\label{lemma3.6}
We have $F\cap \langle U\rangle = U$ and $T\cap \langle U\rangle=\langle V\rangle\subseteq B$.
\end{claim}

\begin{proof}
For the first assertion, let $f\in F\cap \langle U\rangle$, and write $f=u_1\cdots u_k$ for some
$u_1,\ldots,u_k\in U$. Then 
$$f\pi=(u_1\pi)\cdots (u_k\pi)=u_1\cdots u_k=f$$
and so $f\in U$. Therefore 
$F\cap\langle U\rangle\subseteq U$, 
and the converse inclusion is obvious.

For the second assertion, let $t\in T\cap \langle U\rangle$, and write again 
$t=u_1\cdots u_k$ with $u_1,\dots,u_k\in U$.
Choose this product so that $k$ is as small as possible.
Obviously, $k\geq 2$.
We prove that $t\in \langle V\rangle$ by induction on $k$.
For $k=2,3$ we have $t\in V$.
Suppose now $k\geq 4$. Note that both $u_1u_2$ and $u_3\cdots u_k$ must belong to $T$ 
by the first assertion and minimality of $k$.
Thus $u_1u_2\in V$ and, by induction, $u_2\cdots u_k\in\langle V\rangle$,
so that $t\in \langle V\rangle$ as well, as required.
\end{proof}

In what follows, as a technical convenience, we will take $1$ to denote an identity element adjoined to $S$,
for any set $X\subseteq S$ write $X^1=X\cup\{1\}$, and adopt convention that $1\pi=1$.

\begin{claim}
\label{lemma3.7}
Let $n\geq 0$ be arbitrary. Every element $s\in S$ can be represented in the form
\begin{equation}\label{theform}
u_1w_1u_2w_2\cdots u_{k}w_ku_{k+1},
\end{equation}
where $k\geq 0$, $u_1,u_{k+1}\in U^1$, $u_i\in U$ for $i=2,\dots,k$, and $w_i\in \langle Z\pi^n\cup V\rangle$
for $i=1,\dots,k$.
\end{claim}

\begin{proof}
First of all note that since $Z\cup U$ is a generating set for $S$
and $U\pi=U$, we have that $Z\pi^n\cup U=(Z\cup U)\pi^n$ is a
generating set for $S$. 
By Claim \ref{lemma3.6} a product of generators from $U$ of length greater than $1$ can be 
replaced either by a single element from $U$, or by a product of elements from $V$.
\end{proof}

\begin{claim}
\label{lemma3.8}
For any $u_1,u_2\in U^1$ and $w\in Z\cup V$ there exists $n\geq 1$ such that $u_1 (w\pi^k) u_2 \in U\cup B$ for all $k\geq n$.
\end{claim}

\begin{proof}
From $F\pi\cap F= U$ and $F\pi\cap T\subseteq B$ it follows that $F\pi\subseteq U\cup B$.
Combining this with \eqref{eq3}, we see that for every $s\in S$ there exists $n\geq 1$ such that $s\pi^k\in B\cup F$
for all $k\geq n$. Applying this to $s=u_1wu_2$, and remembering that $u\pi^k=u$ for all $u\in U^1$,
yields the result.
\end{proof}

Note that since $V\subseteq \langle U\rangle$, every element of $V$ is fixed by $\pi$.
The sets $U^1$, $Z$ and $V$ are all finite, and so  Claim \ref{lemma3.8} implies that
there exists $M\geq 1$ such that
\begin{eqnarray}
\label{eq5}
&& T\cap\{u_1wu_2:u_1,u_2\in U^1,~w\in Z\pi^M\cup V\}\subseteq B\\
\label{eq6}
&& F\cap\{u_1wu_2:u_1,u_2\in U^1,~w\in Z\pi^M\cup V\}\subseteq U.
\end{eqnarray}

Let us now consider an arbitrary element $t\in T$.
Write $t$ in the form~\eqref{theform}
with $w_i\in \langle Z\pi^M\cup V\rangle$.
Furthermore, choose this decomposition so that 
the length of the corresponding product of generators $U\cup V \cup Z\pi^M$ is as short as possible.
Consider now an arbitrary $w_i$, $i=1,\dots,k-1$. Suppose that its shortest expression as a product of generators
from $Z\pi^M \cup V$ starts with $a\in Z\pi^M\cup V$, and write $w_i=aw_i^\prime$.
From \eqref{eq5}, \eqref{eq6} it follows that $u_ia\in U\cup B$.
But we cannot have $u_ia\in U$, as that would allow us to shorten the expression for $t$.
Hence $u_ia\in B$, and since $w_i^\prime\in  \langle Z\pi^M\cup V\rangle^1\subseteq B^1$,
it follows that $u_iw_i\in B$ for all $i=1,\dots,k-1$.
A similar argument shows that $u_k w_ku_{k+1}\in B$; one just needs to consider both the first and the last factors of $w_k$.
This implies that $t\in B$, and hence $T=B$. But then $S=T\pi=B\pi=B=T$, a contradiction as $T$ is a proper subsemigroup of $S$.
This completes the proof of Theorem \ref{thm1}.
\end{proof}

\section{The proof of the Main Theorem}
\label{sec_proof}

Let $S$ be a finitely generated semigroup, and let $T$ be a hopfian subsemigroup of finite index.
Suppose $\phi : S\rightarrow S$ is a surjective endomorphism of~$S$.

Let $F=S\setminus T$.
Since $\phi$ is onto, for every $k\geq 0$ we must have  
$T\phi^k\supseteq S\setminus F\phi^k$, and so $T\phi^k$ is a cofinite subsemigroup
of $S$; moreover we have $|S\setminus T\phi^k|\leq |F\phi^k|\leq |F|$. 
By~\cite[Corollary 4.5]{Rick}, a finitely generated semigroup has only finitely many cofinite subsemigroups
of any given complement size.
Therefore
there exist $k,r\geq 1$ such that $T\phi^k=T\phi^{k+r}$, 
and hence $T\psi=T\psi^2$, where $\psi=\phi^{(k+1)r}$. 

From $T\psi^2=T\psi$ it follows that $(T\cup T\psi)\psi=T\psi$.
Since $\psi$ is onto, we must have
$$
(S\setminus (T\cup T\psi))\psi\supseteq S\setminus (T\cup T\psi)\psi=S\setminus T\psi.
$$
Now we have
$$
|S\setminus T\psi| \geq |S\setminus (T\cup T\psi)|\geq |(S\setminus (T\cup T\psi))\psi|
\geq |S\setminus T\psi|,
$$
and hence $S\setminus T\psi = S\setminus (T\cup T\psi)$, which in turn implies $T\subseteq T\psi$.

Thus $\psi$ is a surjective endomorphism of $T\psi$, and $T$ is a subsemigroup of finite index mapping onto the whole of $T\psi$.
By Theorem \ref{thm1}, $T$ cannot be a proper subsemigroup, and hence $T=T\psi$.
Thus $\psi\rrestriction_T$ is a surjective endomorphism of $T$, and, since $T$ is hopfian,
$\psi\rrestriction_T$ is actually an automorphism.
Since $\psi$ is a surjection on $S$ and $T\psi=T$, it follows that $F\subseteq F\psi$.
Since $F$ is a finite set it follows that $F\psi=F$, and that $\psi\rrestriction_F$ is a bijection.
Thus $\psi$ as a whole is bijective, and hence so is $\phi$ since $\psi=\phi^{(k+1)r}$.
The Main Theorem has been proved.

\section{A concluding example}
\label{sec_concludeex}

The purpose of this section is to exhibit an example which demonstrates that hopficity is not inherited by cofinite subsemigroups even in the finitely generated case. This is accomplished in Theorem \ref{thm2} at the end of the section.

The construction relies on the notion of S-acts (or actions). 
All actions will be on the right, and to distinguish them from the semigroup operations we will denote the result
of the action of a semigroup element $s\in S$ on an element $x\in X$ by $x\cdot s$.
An $S$-act $X$ can, of course, be viewed as an algebraic structure in its own right, 
with every $s\in S$ inducing a unary operation $x\mapsto x\cdot s$ on $X$.
Therefore the standard algebraic notions -- substructures, homomorphisms, generation, hopficity -- are all meaningful in this context. For a systematic introduction into the semigroup actions see for instance \cite[Section 8.1]{Howie}.

Our first result is well known, but since we have not been able to locate an explicit example in the literature, we give one here for completeness.

\begin{lemma}
\label{lemma1}
The free semigroup of rank $3$ admits a cyclic non-hopfian act.
\end{lemma}

\begin{proof}
Let $F=\langle a,b,c \:|\: \rangle$ be the free semigroup of rank $3$.
Consider the action of generators $a,b,c$ on the set
$$
X=\{ x_i,y_i \st i\in \mathbb{Z}\}\cup \{ z_i\st i\in\mathbb{N}\}\cup \{0\}
$$
given by
\begin{eqnarray*}
&& x_i\cdot a =x_{i+1},\\
&& x_i \cdot b=x_{i-1},\\
&& x_i\cdot c = y_i,\\
&& y_i\cdot a =y_i\cdot b= 0,\\
&& y_i\cdot c =\left\{ \begin{array}{ll} y_i & \mbox{if } i\leq 0,\\
                                                            z_i & \mbox{if } i>0,
                                   \end{array}\right.
\\
&& z_i\cdot a=z_i\cdot b=0,\\
&& z_i\cdot c= z_i,\\
&& 0\cdot a=0\cdot b=0\cdot c=0.
\end{eqnarray*}
This action is shown in Figure \ref{fig1}. Since $F$ is free on $a,b,c$, this action extends to a unique action of $F$ on $X$.
Clearly, this action is generated by $x_0$ (or, indeed, any $x_i$).

Let $\psi : X\rightarrow X$ be defined by
\begin{eqnarray*}
&& x_i\psi= x_{i-1},\\
&& y_i\psi=y_{i-1},\\
&& z_1\psi=y_0,\\
&& z_i\psi=z_{i-1}\ (i>1),\\
&& 0\psi=0. 
\end{eqnarray*}
Effectively, $\psi$ moves all of $x_i,y_i,z_i$ one to the left, except for $z_1$ which it maps to $y_0$, the same as $y_1$.
It is a routine matter to verify that $\psi$ is a surjective, non-injective endomorphism of $X$.
\end{proof}

\begin{figure}
\psset{unit=1mm}
\psset{linewidth=0.3}
\begin{pspicture}(0,5)(120,60)


\pscircle*(30,50){1.5}
\pscircle*(45,50){1.5}
\pscircle*(60,50){1.5}
\pscircle*(75,50){1.5}
\pscircle*(90,50){1.5}

\pscircle*(30,35){1.5}
\pscircle*(45,35){1.5}
\pscircle*(60,35){1.5}
\pscircle*(75,35){1.5}
\pscircle*(90,35){1.5}

\pscircle*(75,20){1.5}
\pscircle*(90,20){1.5}

\psset{arrowsize=1.5}
\psset{arrowlength=1.5}
\psset{arrowinset=0.2}

\multirput*(0,0)(15,0){4}{\psline{->}(31,51)(44,51)}

\multirput*(0,0)(15,0){4}{\psline{<-}(31,49)(44,49)}

\multirput*(0,0)(15,0){5}{\psline{->}(30,48.5)(30,36.5)}

\psline{->}(75,33.5)(75,21.5)
\psline{->}(90,33.5)(90,21.5)

\multirput(0,0)(15,0){3}{\psarc{->}(30,31){3}{100}{80}}
\psarc{->}(75,16){3}{100}{80}
\psarc{->}(90,16){3}{100}{80}

\multirput(0,0)(2,0){3}{\pscircle*(15,50){0.3}}
\multirput(0,0)(2,0){3}{\pscircle*(15,35){0.3}}
\multirput(0,0)(2,0){3}{\pscircle*(95,50){0.3}}
\multirput(0,0)(2,0){3}{\pscircle*(95,35){0.3}}
\multirput(0,0)(2,0){3}{\pscircle*(95,20){0.3}}

\multirput*(0,0)(15,0){4}{\rput(37.5,53){$a$}}
\multirput*(0,0)(15,0){4}{\rput(37.5,47){$b$}}
\multirput*(0,0)(15,0){5}{\rput(27,42.5){$c$}}
\multirput*(0,0)(15,0){3}{\rput(30,25){$c$}}
\rput(72,27.5){$c$}
\rput(87,27.5){$c$}
\rput(75,10){$c$}
\rput(90,10){$c$}

\rput(30,54){$x_{-2}$}
\rput(45,54){$x_{-1}$}
\rput(60,54){$x_{0}$}
\rput(75,54){$x_{1}$}
\rput(90,54){$x_{2}$}

\rput(25,35){$y_{-2}$}
\rput(40,35){$y_{-1}$}
\rput(56,35){$y_{0}$}
\rput(71,35){$y_{1}$}
\rput(86,35){$y_{2}$}

\rput(71,20){$z_{1}$}
\rput(86,20){$z_{2}$}

\end{pspicture}

\caption{A non-hopfian action of $F$ on $X$. The arrows not shown all point to $0$.}
\label{fig1}
\end{figure}
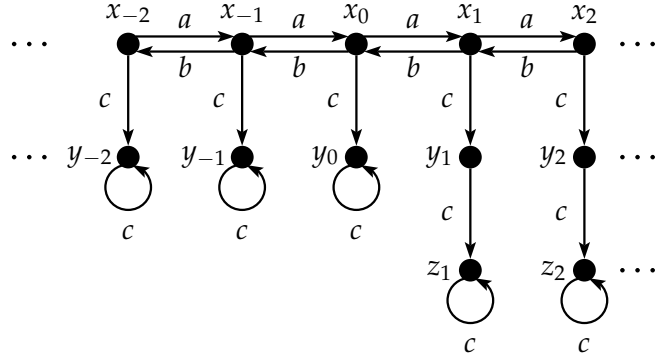

\begin{lemma}
\label{lemma2}
Let $F$ be a free semigroup, and let $X$ be a cyclic $F$-act.
Then there exists an $F$-act $Y$ such that the following hold:
\begin{enumerate}[label=\textup{(\roman*)},partopsep=0mm,topsep=2mm,itemsep=1mm,partopsep=0mm]
\item
\label{Y1}
$X$ is a subact of $Y$;
\item
\label{Y2}
$|Y\setminus X| =1$;
\item
\label{Y3}
$Y$ is hopfian.
\end{enumerate}
\end{lemma}

\begin{proof}
Suppose that $F=\langle A\:|\: \rangle$, and suppose $X$ is generated by $x_0$, i.e. $x_0\cdot\nobreak F^1=\nobreak X$.
Let $Y=X\cup\{y_0\}$, where $y_0\not\in X$.
Extend the action of $F$ on $X$ to an action on $Y$ by setting
$$
y_0\cdot a = x_0 \ (a\in A).
$$
Assertions \ref{Y1} and \ref{Y2} are clear.
To verify that $Y$ is hopfian, let $\psi : Y\rightarrow Y$ be any surjective endomorphism.
Since $Y\cdot F=X$ (i.e. $y_0$ is the only element of $Y$ which has no arrows coming into it)
it follows that $x\psi\neq y_0$ for all $x\in X$. This, combined with $\psi$ being onto, implies $y_0\psi=y_0$.
Since the $F$-act $Y$ is generated by $y_0$, it readily follows that $\psi$ must be the identity mapping, and so $Y$ is indeed hopfian.
\end{proof}

We now introduce a semigroup construction that we then use to build our desired example.
The ingredients for the construction are a semigroup $S$ and an $S$-act $X$ (with $S\cap X=\emptyset$).
The new semigroup, which we denote by $S[X]$, has the carrier set $S\cup X$; the multiplication in $S$ remains the same, while
for $s\in S$, $x,y\in X$ we define
$$
sx=x,\ xs=x\cdot s,\ xy=y.
$$
It is a routine matter to check that $S[X]$ is indeed a semigroup.

\begin{lemma}
\label{lemma3}
Let $S$ be a semigroup, let $X,Y$ be two $S$-acts, and let $\psi : X\rightarrow Y$ be a homomorphism of $S$-acts.
Define a mapping $\phi :  S[X]\rightarrow S[Y]$ by
$\phi= 1_S\cup \psi$, where $1_S$ is the identity mapping on $S$.
Then $\phi$ is a (semigroup) homomorphism.
Moreover, $\phi$ is surjective (respectively injective, bijective) if and only if $\psi$ is surjective (resp. injective, surjective).
\end{lemma}

\begin{proof}
For $s,t\in S$ and $x,y\in X$ we have
\begin{eqnarray*}
&&
(st)\phi=st=(s\phi)(t\phi),\\
&&
(sx)\phi=x\phi=x\psi=s(x\psi)=(s\phi)(x\phi),\\
&&
(xs)\phi=(x\cdot s)\phi=(x\cdot s)\psi=(x\psi)\cdot s=(x\phi)s=(x\phi)(s\phi),\\
&&
(xy)\phi=y\phi=y\psi=(x\psi)(y\psi)=(x\phi)(y\phi).
\end{eqnarray*}
The final three assertions are obvious.
\end{proof}

\begin{lemma}
\label{lemma4}
Let $F=\langle A\:|\:\rangle$ be a free semigroup of finite rank, let $X$ be an $F$-act, and suppose that
$\phi : F[X]\rightarrow F[X]$ is a surjective endomorphism. Then:
\begin{enumerate}[label=\textup{(\roman*)},partopsep=0mm,topsep=2mm,itemsep=1mm,partopsep=0mm]
\item
\label{4i}
$\phi\rrestriction_F$ is an automorphism of $F$.
\item
\label{4ii}
$X\phi=X$.
\end{enumerate}
\end{lemma}

\begin{proof}
\ref{4i}
All elements of $X$ are idempotents, while $F$ has no idempotents;
hence $X\phi\subseteq X$.
Since $\phi$ is onto it follows that $F\subseteq F\phi$.
Now let
\begin{eqnarray*}
&&
A_F=\{ a\in A\st a\phi\in F\},\\
&&
A_X=\{ a\in A\st a\phi\in X\}.
\end{eqnarray*}
Since $X$ is an ideal of $F[X]$, it follows that $(F^1A_XF^1)\phi\subseteq X$.
Again, since $\phi$ is onto we must have $\langle A_F\rangle\phi=F$.
But $\langle A_F\rangle$ is a free subsemigroup of $F$ of rank $|A_F|$.
Since $A$ is finite it follows that $A_F=A$ and $A_X=\emptyset$.
Hence $\phi\rrestriction_F$ is a surjective endomorphism of $F$, and indeed an
automorphism since $F$ is hopfian.

\ref{4ii}
We have already proved $X\phi\subseteq X$. The assertion now follows from $F\phi=F$ and $\phi$ being
surjective.
\end{proof}

\begin{lemma}
\label{lemma5}
Let $F$ be a free semigroup of finite rank, and let $X$ be an $F$-act.
Suppose $\phi : F[X]\rightarrow F[X]$ is a surjective endomorphism with $\phi\rrestriction_F=1_F$.
Then $\phi\rrestriction_X$ is a surjective endomorphism of the $F$-act $X$.
\end{lemma}

\begin{proof}
By Lemma \ref{lemma4} \ref{4ii} we have that $\phi\rrestriction_X$ maps $X$ onto itself.
Furthermore, for $x\in X$ and $s\in F$, we have
$$
(x\cdot s)\phi\rrestriction_X=(xs)\phi=(x\phi)(s\phi)=(x\phi)s=(x\phi\rrestriction_X)\cdot s,
$$
i.e. $\phi\rrestriction_X$ is an $F$-act endomorphism.
\end{proof}

\begin{lemma}
\label{lemma6}
Let $F$ be a free semigroup of finite rank, and let $X$ be an $F$-act.
The semigroup $F[X]$ is hopfian if and only if $X$ is a hopfian $F$-act.
\end{lemma}

\begin{proof}
($\Rightarrow$)
Suppose $F[X]$ is hopfian, and let $\psi : X\rightarrow X$ be a surjective endomorphism of $F$-acts.
Using Lemma \ref{lemma3}, there is a surjective endomorphism $\phi : F[X]\rightarrow F[X]$ such that
$\phi\rrestriction_X=\psi$.
Since $F[X]$ is hopfian, $\phi$ is injective, and hence $\psi$ is injective as well.

($\Leftarrow$)
Suppose $X$ is a hopfian $F$-act, and let $\phi : F[X]\rightarrow F[X]$ be a surjective endomorphism.
By Lemma \ref{lemma4}, the mapping $\phi\rrestriction_F$ is an automorphism of $F$.
Since the automorphism group of $F$ is isomorphic to the finite symmetric group $S_r$ (where $r$ is the rank of $F$),
there exists $n\in\mathbb{N}$ such that $(\phi\rrestriction_F)^n=1_F$.
By Lemma \ref{lemma5}, applied to the mapping $\phi^n$, we have that $(\phi\rrestriction_X)^n$ is a surjective endomorphism
of the $F$-act $X$.
But $X$ is hopfian, and hence $(\phi\rrestriction_X)^n$ is injective.
It follows that $\phi\rrestriction_X$, and indeed $\phi$ itself, are injective, and so $F[X]$ is hopfian.
\end{proof}

\begin{thm}
\label{thm2}
There exists a finitely generated hopfian semigroup $S$ which contains a non-hopfian subsemigroup $T$ with $|S\setminus T|=1$.
\end{thm}

\begin{proof}
Let $F=\langle a,b,c \:|\: \rangle$ be the free semigroup of rank $3$, and let $X$ be a cyclic, non-hopfian $F$-act, guaranteed by Lemma \ref{lemma1}.
Extend $X$ to a cyclic hopfian $F$-act $Y$ with $|Y\setminus X|=1$, as in Lemma \ref{lemma2}.
Let $S=F[Y]$, $T=F[X]$.
Clearly, $T\leq S$ and $|S\setminus T|=1$.
By Lemma \ref{lemma6} we have that $S$ is hopfian, while $T$ is not.
Finally, $S$ is finitely generated: indeed $S=\langle a,b,c,y_0\rangle$, where $y_0$ is any generator of the $F$-act $Y$.
\end{proof}

\section{Concluding remarks}
\label{remarks}

If we perform the construction as described in the proof of Theorem \ref{thm2}, starting from the act exhibited in
Lemma \ref{lemma1} it is relatively easy to see that the resulting semigroups $S$ and $T$ are not finitely presented.
In fact, it seems unlikely that this construction will ever give a finitely presented example. This leaves open the question of 
\textit{existence of a finitely presented hopfian semigroup $S$ containing a non-hopfian subsemigroup $T$ of finite Rees index.}
Analogous question where $S$ is required to have a \textit{finite confluent noetherian rewriting system} (see \cite{BookOtto}) is also
of interest, especially in the light of \cite{wongta}.

Although Rees index is in many ways analogous to the group-theoretic index, it obviously does not generalise the latter.
A viable common generalisation -- Green index -- has recently been proposed in \cite{Bob}.
This leads us to ask: \textit{Is it true that if a finitely generated semigroup $S$ has a hopfian subsemigroup $T$ of finite Green index then $S$ itself must be hopfian?}
If the answer is positive, the proof of this fact would most likely incorporate elements of both Hirshon's original argument, and our 
considerations in Sections \ref{sec_reesendo}, \ref{sec_proof}.

We close the paper with the following open problem, which in fact
stimulated the first author to think about hopficity:
\textit{Is hopficity decidable for one relation semigroups (or one relator groups)?}
This problem appears to be quite difficult at present.
In fact, even the seemingly
 easier related question of \emph{deciding residual finiteness} is still open. 
The reader should recall a
classical theorem of Malcev \cite{Malcev} (see also \cite[Theorem IV.4.10]{Lyndon}) that a finitely generated residually
finite group (or semigroup) is hopfian, and consult \cite{Lallement} and~\cite{Sapir} for some relevant information.
We finish by showing that
this question is not vacuous, and filling in an obvious and surprising gap in the literature:

\begin{example}
The  one relation semigroup $S=\mathrm{Sg}\langle a,b \:|\: abab^2ab=b\rangle$ is non-hopfian.
To verify this, first note that
\begin{equation*}
abab^3=abab^2\cdot abab^2ab=abab^2ab\cdot ab^2ab=bab^2ab.
\end{equation*}
It easy to check that the rewriting system $\{abab^2ab\to b, abab^3\to
bab^2ab\}$ is confluent and noetherian and so defines $S$. Notice
that
\begin{eqnarray*}
a\cdot bab\cdot a\cdot (bab)^2\cdot a\cdot bab &=& abab\cdot abab^2ab\cdot abab\\
&\to& abab^2ab\cdot ab\\
&\to& bab.
\end{eqnarray*}
This means that the assignment $a\mapsto a$, $b\mapsto bab$ lifts
to an endomorphism $\phi$ of $S$. Since $a\cdot bab\cdot bab=b$, the endomorphism
is onto. Under $\phi$ we obviously have $ab^2\mapsto b$ and
so $ab^2a^2b^2=ab^2\cdot a\cdot ab^2\mapsto bab$. But, by our
rewriting system, $ab^2a^2b^2\neq b$ and so $\phi$ is not bijective. 
\end{example}

\begin{flushleft}
School of Mathematics and Statistics\\
University of St Andrews\\
St Andrews KY16 9SS\\
Scotland, U.K.\\
\smallskip
\texttt{victor@mcs.st-and.ac.uk, nik@mcs.st-and.ac.uk}
\end{flushleft}

\end{document}